\documentclass[pdflatex,sn-mathphys]{sn-jnl}
\usepackage{enumerate}
\usepackage[shortlabels]{enumitem}
\usepackage{amsmath}

\DeclareMathOperator*{\argmin}{arg\,min}
\DeclareMathOperator*{\pc}{\mathbf{PC}}

\DeclareMathOperator*{\geo}{\mathbf{Geo}_{\tau}}

\DeclareMathOperator{\spt}{spt}
\DeclareMathOperator{\hess}{Hess}
\DeclareMathOperator{\lip}{Lip}
\jyear{2021}%

\theoremstyle{thmstyleone}%
\newtheorem{theorem}{Theorem}
\newtheorem{proposition}[theorem]{Proposition}%
\newtheorem{lemma}[theorem]{Lemma}
\newtheorem{corollary}[theorem]{Corollary}

\theoremstyle{thmstyletwo}%
\newtheorem{remark}{Remark}%

\theoremstyle{thmstylethree}%
\newtheorem{definition}{Definition}%
\begin{document}

\title[Aggregation equation via $d_2$-gradient flows]{Minimizing movement scheme for intrinsic aggregation on compact Riemannian manifolds }

\author{\fnm{Joaquín Sánchez García}}\email{joaqsan@math.utoronto.com}


\abstract{Recently solutions to the aggregation equation on compact Riemannian Manifolds have been studied with different techniques. This work demonstrates the small time existence of measure-valued solutions for suitably regular intrinsic potentials. The main tool is the use of the minimizing movement scheme which together with the optimality conditions yield a finite speed of propagation. The main technical difficulty is non-differentiability of the potential in the cut locus which is resolved via the propagation properties of geodesic interpolations of the minimizing movement scheme and passes to the limit as the time step goes to zero.}

\keywords{minimizing movement scheme, aggregation, sphere, compact manifolds, Wasserstein flow, intrinsic interaction potential}

\maketitle
\tableofcontents
\markboth{}{}
\section{Introduction}\label{sec1}

We consider an aggregation equation on a smooth, connected, compact Riemannian manifold $(M,g)$ in which the evolution of the density $\mu_t$ of population (or particles) is described by 
\begin{equation} \label{aggregation}
 \partial_t \mu_t - \nabla_M \cdot ((\nabla_M (W \ast \mu_t))\mu_t ) = 0
\end{equation}
where $W \ast \mu_t(x) = \int_{M} W(x,y) d\mu_t (y) $ is the convolution operator for a potential function $W: M \times M \to \mathbb{R}$. Further,  we assume that the interaction happens only through intrinsic distance, i.e. $W(x,y) = h(d(x,y))$ where $h: \mathbb{R} \to \mathbb{R}$ is twice continuously differentiable and $d(x,y)$ denotes the Riemannian distance between $x$ and $y$ and is non-decreasing for big enough distances. \\
One of the main problems of models like \eqref{aggregation} is the possible non-differentiability of potentials on the cut locus, there the Riemannian distance fails to be differentiable. To overcome the failure of differentiability of the interaction potential we first prove that the speed of propagation of the minimizing movement scheme is finite (Proposition \ref{finitespeed}), which intuitively says the spread of the particles is slow enough to not instantly fall into the cut locus. \\
Model \eqref{aggregation} has received a lot of attention in recent years. A very similar model has been considered in \citep{Slepcev} where the authors used a related technique motivated by the minimizing movement scheme but depend on the euclidean distance of the ambient space into which the manifold is embedded. The idea of using a generalization of the projection onto the tangent space allows the authors to overcome some difficulties by using the information from the Euclidean setting and also allowing them to conclude stability of the model. \\
Another approach can be found in \citep{Fetecau} and \citep{FetecauPatacchini}, here the authors solve the existence and uniqueness question for weak solutions of model \eqref{aggregation} using the theory of partial differential equations and Lipschitz-coefficients theorems. This approach is very fruitful and motivated the work presented here. In contrast we aim to obtain similar conclusions by only looking at the measures involved. Although both approaches are completely intrinsic (in the sense that they depend only on Riemannian distance) the goal of this document is to adapt the theory of Wasserstein gradient flows as presented in \citep{Ambrosio}. The main difficulty is that the presence of the cut locus stops us from applying directly the theory of Wasserstein Gradient flows as the interaction can not be shown to be globally $\lambda$-convex. We overcome this problem by forcing a regularizing Wasserstein term and noting that optimality conditions control the distance of transport for every timestep $\tau$. \\ An introductory treatment of the Wasserstein gradient flow for the interaction energy on $\mathbb{R}^n$ can be found in Chapter 8 of \citep{Villani}. \\
In the $\mathbb{R}^n$ case more work has been done. Recently, for the case of interaction potentials that can be written as power laws, it was shown in \citep{McCannCameron} that $d_{\infty}$-local minimizers for the mildly repulsive regime concentrate on  a $n$-simplex. Another notable result in the $\mathbb{R}^n$ case is \citep{FigalliCarrillo}, there it was shown that for specific interaction potentials concentration can be immediate. \\
Our approach is specially interesting for applications, as it leads the way to different types of numerical algorithms using the algorithms for optimal transportation (see \citep{CuturiPeyre}). This could open a new line of investigation to compare the efficiency of algorithms derived from PDE-approximation methods against optimal transport based methods. \\
To summarize the paper, we prove that for small times, solutions of the aggregation equation \eqref{aggregation} exist whenever the interaction is intrinsic and satisfies \ref{h0}-\ref{hisincreasing}. To this end we use the minimizing movement scheme in which the Euler-Lagrange conditions for optimality allow us to upgrade the regularity properties of Kantorovich potentials (in the support of the measures) via non-smooth analysis. This analysis shows that the minimizing movement scheme does not immediately move to the cut locus from which one can deduce several properties of the limiting measure. 

\section{Preliminaries and precise formulation of the problem}
Let $(M,g)$ be a smooth, connected, compact $n$-dimensional Riemannian manifold. Let $(x^1,x^2, \cdots, x^n)$ be local coordinates, we denote by $T_pM$ the tangent space at $p \in M$ and let $g_{ij}$ be the metric coordinates. For the canonical basis $ \{ \frac{\partial}{\partial x^1},\frac{\partial}{\partial x^2}, \cdots , \frac{\partial}{\partial x^n} \}$ the gradient of a scalar function on $M$ is given by 
$$ \langle \nabla_M f(p), v \rangle_p = df_p(v) $$
for every $v \in T_pM$, where $df_p$ is the differential of $f$ at $p \in M$ and $\langle \cdot , \cdot \rangle_p$ denotes the inner product in $T_pM$. Hence, by using the local coordinates, 
$$ \nabla_M f = g^{ij} \frac{\partial f}{\partial x^i} \frac{\partial}{\partial x^j}. $$ 
For a tangent vector field $F = F^i \frac{\partial}{\partial x^i}$, we define the divergence,
$$ div(F) = \frac{1}{\sqrt{\det g}} \sum_{i=1}^n \frac{\partial }{\partial x^i} \left( \sqrt{\det g} F^i  \right).$$
If $f:M \to \mathbb{R}$ is differentiable, we define the Hessian, $\hess f$ of $f$ at $p \in M$ as the linear operator
 $ \hess f:T_pM \to T_pM $ via the formula 
 $$ \hess f (Y) = \nabla_{Y} (\nabla_M f) $$ 
 for $ Y \in T_pM$ and $\nabla_Y$ denoting the covariant derivative along $Y$, see \citep{DoCarmo}.
The standard volume in local coordinates is
$$ dvol =  \sqrt{\det g} dx^1 \wedge dx^2 \wedge \cdots  \wedge dx^n. $$
For given $p \in M$ the cut locus at $p$, denoted $Cut(p)$ denotes the set of points on $M$ that can not be linked to $p$ by any extendable geodesic. The Cut locus is the subset of $M \times M = \{ (x,y) \in M \times M : y \in Cut(x) \}$. \\
For $p \in [1,\infty)$ denote by $\mathcal{P}_p(M)$ the set of probability measures with $p$-finite moment, and $\mathcal{P}_{ac}^p(M)$ the subset of $\mathcal{P}_p(M)$ of measures absolutely continuous with respect to $dvol$. For $\mu,\nu \in \mathcal{P}_p(M)$ we define the Wasserstein-$p$ metric via
\begin{equation}\label{WassersteinMetric}
 d_p (\mu, \nu) = \left( \inf_{\pi \in \Pi(\mu,\nu)} \int_{M \times M} d(x,y)^p d\pi(x,y) \right)^{1/p},
 \end{equation}
where $\Pi(\mu,\nu)$ denotes the set of probability measures on $M\times M$ whose marginals are $\mu$ and $\nu$ respectively and again $d(x,y)$ denotes the Riemannian distance between $x$ and $y$. Recall:
\begin{theorem}(Optimal Transportation on Riemannian manifolds) \label{McCann} \\ In a smooth, connected Riemannian manifold, if $\mu,\nu$ are compactly supported measures on $M$ and $\mu$ is absolutely continuous with respect to Riemannian volume, then considering the cost function $d(x,y)^2/2$, there exists an optimal transport map $T$, transporting $\mu$ onto $\nu$ determined uniquely $\mu$ a.e. by 
$$ T(x) = \exp_x(-\nabla \phi (x)),$$
where $\phi$ is some $d^2/2$-concave function. 
\end{theorem}
The proof can be found in \cite{McCann} and it is very standard in optimal transport literature, hence omitted. Another important result for us is the characterization of the $d_1$ metric:
\begin{theorem}(Kantorovich Rubinstein) \\
In a smooth, connected Riemannian manifold $M$, let $\mu, \nu \in \mathcal{P}_1(M)$ then 
\begin{equation}
    d_1(\mu,\nu) = \sup_{\lvert \lvert \psi \rvert \rvert_{Lip} \leq 1} \left\{ \int \psi d(\mu - \nu) \right\},
\end{equation}
where $\lvert \lvert \psi \rvert \rvert_{Lip}$ denotes the Lipschitz constant of $\psi$.
\end{theorem}
See for example \citep{VillaniOldAndNew} Theorem 5.10 particular case 5.16. \\
Finally, we recall that on a compact space the Wasserstein $p$-metrics are ordered, if $ p_1 \geq p_2 $ then $d_{p_2}(\mu,\nu) \leq d_{p_1}(\mu,\nu)$ whenever $\mu,\nu \in \mathcal{P}_{p_2}(M)$, see \citep{Villani} Chapter 7.
We start setting up the problem by defining what we mean by a solution. We are interested in measure-valued solutions to the aggregation model. 
\begin{definition}(Measure-valued solutions) \\
Given $T \in [0,\infty)$ we say that $\{\mu_t\}_{t \in [0,T]}$ is a measure-valued solution to the aggregation model \eqref{aggregation} with potential function $W: M \times M \to \mathbb{R}$ if for every test function, $\phi \in C_{c}^{\infty}((0,T) \times M )$ we have
\begin{equation} \label{measure-sol}
 \int_0^T \int_{M} \partial_t \phi(t,x)  + \langle \nabla_M \phi(t,x), \nabla_M (W \ast \mu_t )(x) \rangle_x  d\mu_t(x)  dt = 0,
\end{equation}
where $C_c^{\infty}$ denotes smooth functions with compact support. 
\end{definition}
\subsection{Assumptions on the potential}\label{assumptions}
Equation \eqref{aggregation} may have no solutions if the potential function is not appropriate.
Given that our goal is to study measure-valued solutions of the equation, we are going to use the energy $E_W$ associated to a potential function $W: M \times M \to \mathbb{R}$ given by \begin{equation} \label{defEW}
E_W(\rho) = \frac{1}{2} \int_M \int_M W(x,y) d\rho(x) d\rho(y).
\end{equation}
We first assume that the interaction depends only through the the Riemannian distance between points so

\begin{enumerate}[start=0,label={(\bfseries W\arabic*)}]
\item \label{h0} Without loss of generality, we assume $h(0) = 0.$ 
\item \label{Wishd} $W(x,y) = h(d(x,y)^2)$ where  $h: [0,\infty) \to \mathbb{R}$ is a twice continuously differentiable on $(0,diam(M)) $ function, further assume $h'(0)$ exists (as the limit from the right).
Namely, we assume $ h \in C^{1}([0,diam(M)]) \cap C_{loc}^2((0,diam(M))$.
\item \label{hisincreasing} There exists $ 0 < r_h < diam(M) $ such that $h$ is non-decreasing on $[r_h,diam(M))$.
\end{enumerate}
\begin{remark}\label{WLip} 
Note that by triangle inequality $ x \to d(x,p)$ is a Lipschitz function with constant $1$ for all $p \in M$. Because $h' $ is continuous on $[0,diam(M)]$, which is compact $h$ is Lipschitz and hence, for every fixed $y$ the interaction potential $ W_y(x) = W(x,y) = h(d(x,y)^2)$ is also Lipschitz (as a function of $x$). Denoting by $\lip$ the Lipschitz constant, for every $y \in M$, the potential $W$ satisfies $\lip(W_y) \leq 2 \lip(h)  diam(M)$. For this reason we  definee  $L:=2 \lip(h)  diam(M) $.\\
Further, Rademacher's Theorem ensures that for every $y \in M$ the function $W_y$ is differentiable $dvol$-almost everywhere. In general this conclusion won't be enough as some of the measures involved may not be absolutely continuous. 
\end{remark}

\subsection{Non-differentiability of the potential}\label{nondiffW}
In order to show that the limit of the minimizing movement scheme (\ref{minimizingmovements}) solves the aggregation equation, we will need the potential function to be differentiable. So far, our interaction potential with assumptions \ref{h0}-\ref{Wishd} fails to be differentiable at the cut locus. We will use the finite speed of propagation (Proposition \ref{finitespeed}) to ensure the potential is differentiable in the support of the measures involved in the minimizing movement scheme. \\
Note that assumptions \ref{h0}-\ref{Wishd} guarantee: \begin{itemize}
    \item The energy $E_W$ is proper ($ \{ \rho \in \mathcal{P}(M): E_{W}(\rho) < \infty \} \neq \emptyset $).
    \item The energy $E_W$ is lower semi-continuous.
\end{itemize}
\begin{remark}\label{Lowerboundk}
As a consequence of compactness of $M$, $W$ is bounded from below by a constant, that we will denote by $k$, that is 
$$ \inf_{(x,y) \in M \times M} W(x,y) \geq k. $$
\end{remark}
The approach is to use the so called ``Minimizing Movement Scheme'' from \cite{JKO},\cite{Ambrosio},\cite{DeGiorgi},\cite{ATW}.
The scheme consists in taking a time step of size $\tau > 0$ to balance the contribution of the original energy (in our case $E_W$) and a term that penalizes moving away from the previous configuration. The minimizer in this scheme approximates a step in the direction of steepest descent, getting more accurate as $\tau $ approaches $0$.  

\begin{definition}(Minimizing movement scheme on $(\mathcal{P}^2(M),d_2)$ for $E_W$)\\
Let $\mu_0 \in \mathcal{P}^2(M)$ be fixed, for $\tau > 0$ if it is possible to define a sequence $\{\mu_k^{\tau} \}$ of probability measures such that
\begin{equation}\label{minimizingmovements}
    \mu_{k+1}^{\tau} \in \argmin \left\{ E_{W}(\rho) + \frac{1}{2\tau} d_2^2(\rho,\mu_{k}^{\tau}) : \rho \in \mathcal{P}^2(M) \right\};
\end{equation}
we call $\{\mu_{k}^{\tau}\}$ a sequence of the minimizing movement scheme for $E_W$ at level $\tau$.
\end{definition}
\begin{remark}(Existence of minimizing movement scheme) \\
Note that the functional in \eqref{minimizingmovements} is lower semi-continuous with the assumptions \ref{h0}- \ref{Wishd}  it is easy to show that $\mu \to E_W(\mu)$ is lower semi-continuous and the distance to any given measure is lower semi-continuous by triangle inequality. Hence, the lower semi-continous functional on a compact set achieves a minimum yielding existence of a sequence $\{\mu_{k}^{\tau}\}_{k \in \mathbb{N}}$ for every $\tau > 0$. Existence of minimizers of the scheme does not ensure the model \eqref{aggregation} will be solved by any time interpolation, the rest of the work is dedicated to interpolating the measures in a continuous way and showing the limiting measure solves the aggregation equation. 
\end{remark}
\begin{remark} \label{interpolation}
Because our goal is to solve \eqref{measure-sol}, we need time interpolation of the sequences in the minimizing movement scheme. We denote by $\pc_{\tau}(\{\mu_k^{\tau}\}) (t)$ the piecewise constant interpolation such that $\pc_{\tau}(\{\mu_k^{\tau}\}) (t) = \mu_{k}^{\tau}$ if $ t \in [k\tau, (k+1)\tau) $. Finally, we define the geodesic interpolation by the following formula, for $ t \in [k\tau,(k+1)\tau)$

$$ \mathbf{Geo}_{\tau}(\{ \mu_{k}^{\tau} \}) = \exp_x \left( \left(\frac{(k+1)\tau - t}{\tau} \right) \nabla \phi_{k,k+1}^{c}  \right)_{\#} \mu_{k+1}^{\tau}, $$
where $\phi_{k,k+1}$ is the Kantorovich potential from $\mu_{k}^{\tau}$ to $\mu_{k+1}^{\tau}$ and $\phi_{k,k+1}^{c}(x)$ is the $c-$transform (or infimal convolution) of $\phi_{k,k+1}$ given by
$$ \phi_{k,k+1}^c(x)= \inf_{z \in M} \left\{ \frac{d(x,z)^2}{2} - \phi_{k,k+1}(z)\right\}.$$
\end{remark}
\begin{remark}
A function $f$ is called $c-$concave if it is not $-\infty$ and is the $c-$transform of another function. The fundamental theorem of optimal transport says that optimal plans in \eqref{WassersteinMetric} are supported in $c-$subdifferentials of $c-$concave functions, see \cite{Users} Theorem 1.13. \\
Given measures $\mu,\nu \in \mathcal{P}(M)$ it is not necessarily true that the Kantorovich potential $\phi_{\mu,\nu}$  from $\mu$ to $\nu$ is differentiable in $\spt(\mu)$. Some conditions are always necessary (e.g. absolute continuity of the source measure) to ensure differentiability. The optimality criterion (Euler-Lagrange) of Lemma \ref{optimality} will yield differentiability as we will see in Proposition \ref{phidifferentiable}. 
\end{remark}
\subsubsection{Square estimates on the Wasserstein norms}
\begin{theorem}\label{squareest} 
Let $\{ \mu_k^{\tau} \}$ be a minimizing movement scheme for $E_W$ as in \eqref{minimizingmovements}, i.e.  $\{ \mu_k^{\tau} \}$ satisfies
$$ \mu_{k+1}^{\tau} \in \argmin_{\mu \in \mathcal{P}^2(M)} \left\{ E_W(\mu) + \frac{d_2(\mu,\mu_k)^2}{2\tau} \right\}. $$ Then there exists a constant $C > 0$ independent of $\tau$ such that 
\begin{equation} \label{squareestimate}
 \sum_{k=0}^{\infty} \frac{d_2(\mu_k^{\tau},\mu_{k+1}^{\tau})^2}{\tau} \leq C.
\end{equation}
\end{theorem}
\begin{proof}
 The proof is standard and can be found in \citep{Villani}, presented here for completeness. Note that the optimality condition of $ \mu_{k+1}^{\tau}$ implies
$$ E_W(\mu_{k+1}^{\tau}) + \frac{d_2(\mu_{k}^{\tau},\mu_{k+1}^{\tau})^2}{2\tau} \leq E_W(\mu_k^{\tau}). $$ 
Hence, given that $E_W$ is proper, the sequence gives finite values for $E_W$ and so
$$   \frac{d_2(\mu_{k}^{\tau},\mu_{k+1}^{\tau})^2}{2\tau} \leq E_W(\mu_k^{\tau}) - E_W(\mu_{k+1}^{\tau}). $$
By summing all the terms, we get a telescopic sum on the right hand side, and the fact that $E_W$ is bounded from below by $k$ (Remark \ref{Lowerboundk}) gives
$$ \sum_{k=0}^{\infty} \frac{d_2(\mu_{k}^{\tau},\mu_{k+1}^{\tau})^2}{2\tau} \leq E_W(\mu_0) - k, $$
where $k$ is the lower bound of $E_W$ obtained by compactness (Remark \ref{Lowerboundk}), hence putting $C := 2(E_W(\mu_0)-k)$ gives the claim as it is finite and independent of $\tau$. 
\end{proof}

\section{Existence of solutions}\label{sec2}
We aim to analyze measure-valued solutions to \eqref{aggregation}. The main technical difficult is dealing with the fact that an interaction potential $W(x,y) = h(d(x,y)^2)$ may not be differentiable at the cut locus $Cut \subseteq M \times M$ (see section \ref{nondiffW}). The aggregation equation together with the minimizing movement scheme \eqref{minimizingmovements} will be shown to satisfy a finite-speed of propagation Proposition \ref{finitespeed}. This means that if we start with a probability measure concentrated away from the cut locus, we can apply the $d_2$-gradient flow method to generate solutions to the equations for small times. \\
As we are going to use the gradient of the interaction potential, we must ensure $W$ remains differentiable. 
We are going to prove that if the initial measure $\mu_0$ is concentrated away from the cut locus, solutions (for small time) exist in measure sense. The idea is that the minimizing movement scheme will not instantly move to the cut locus, it needs time to spread.
\begin{definition}\label{distancecut}(Distance to Cut)\\
Let $(M,g)$ be smooth, compact, connected Riemannian manifold, for $\mu \in \mathcal{P}_{c}(M)$ we define the distance to cut locus, $\delta_{\mu}$, as the distance between every pair on the support to the cut locus, i.e.
\begin{equation}
    \delta_{\mu} := \inf_{\substack{x,y \in \spt(\mu) \\  (x',y') \in Cut}} \{ d(x,x') + d(y,y')\}.
\end{equation} 
\end{definition}
\begin{theorem} (Local existence of measure solutions to the aggregation equation) \label{existenceanduniqueness}\\ 
Given $\mu_0 \in \mathcal{P}_{ac}^2(M)$ let $\delta_{\mu}$ be the distance to cut as in Definition \ref{distancecut}, if $\delta_{\mu} > 0$ and $L$ denotes the Lipschitz constant of $W$ (from \ref{Wishd}); under \ref{h0}-\ref{Wishd}, for every $ 0 < T < \frac{\delta_{\mu}}{2L}$  there exists a sequence from the minimizing movement scheme for $E_W$ at level $\tau > 0$ starting at $\mu_0$ such that as $\tau \to 0$ the geodesic interpolation $\mathbf{Geo}_{\tau}(\{\mu_{k}^{\tau} \})(t)$ converges in $d_2$-metric to a path $ \mu(t) $ which is a measure valued solution (in the sense of \eqref{measure-sol}) to the aggregation equation on $M$ \eqref{aggregation} up to time $T$. 

\end{theorem}
The proof of Theorem \ref{existenceanduniqueness} is the main goal of this document and will occupy the rest of the article. We will ensure the convergence of the minimizing movement scheme using a general version of the Arzela-Ascoli that can be found in \citep{Bourbaki}.
\begin{theorem} \label{MetricAA} (General version of Arzela-Ascoli) \\
Let $ X$ be a topological space and $Y$ a metric space, let $H$ be an equicontinuous family of  functions from $X$ to $Y$ such that for every $x \in X$, $H(x) := \{ h(x): h \in H\}$ is relatively compact in $Y$. Then $H$ is relatively compact with respect to the compact topology.
\end{theorem} 
For a proof see Corollary 1 to Theorem 2, section V of \citep{Bourbaki}
\begin{corollary}(Existence of a Limiting path) \label{corollaryexists} \\
Fix $T > 0$, suppose that $\geo (\{\mu_{k}^{\tau})\}(t)$ is defined for all $ \tau \in (0,1]$ and for all $ t \in [0,T]$.Then there exists a subsequence $\tau_n$, with $\tau_n \to 0$ and a curve $\mu:[0,T] \to (\mathcal{P}_2(M), d_2)$ such that 
\begin{equation}\label{uniformont}
\sup_{t \in [0,T]} d_2(\geo (\{\mu_{k}^{\tau_n})\}(t),\mu(t)) \to 0 \: \: (\text{as } n \to 0).
\end{equation}
\begin{proof}
Note that Lemma \ref{abscontleveltau} (proved in the next section) shows uniform equicontinuity of the family, as the Hölder constant does not depend on $\tau$. For every $t \in [0,T]$, the family $\{ \geo (\{\mu_{k}^{\tau})\}(t) \}$ is tight by Prokhorov's theorem and because $M$ is compact, it is also $d_2$ relatively compact. Hence $H = \{ \geo (\{\mu_{k}^{\tau}\}): [0,T] \to \mathcal{P}_2(M) \}_{\tau}$ satisfy the hypothesis of Theorem \ref{MetricAA} and the result follows. 
\end{proof}
\end{corollary} 
\subsection{Properties of the limiting measure path}
By Corollary \ref{corollaryexists} we know that as long as we can define the geodesic interpolation up to time $T$, we obtain the existence of a limiting path $\mu(t)$. We have yet to show that this limiting path $\mu(t)$ satisfies \eqref{aggregation}, for which we will work with the Euler-Lagrange conditions of the minimizing movement scheme \eqref{minimizingmovements}.
\subsection{Continuity and optimality }

\begin{definition}(Absolute continuity in $(\mathcal{P}^2(M),d_2)$) \label{a.c.} \\
We say that a curve $t \to \rho_t$ mapping $(a,b)$ to $ (\mathcal{P}^2(M),d_2) $ is absolutely continuous if there exists an integrable (w.r.t Lebesgue)  function $g:(a,b) \to \mathbb{R}$ such that
$$ d_2(\rho_t,\rho_s) \leq \int_s^t g(r) dr. $$
And we say it is $p$-absolutely continuous if $g$ is $L^p$-integrable.
\end{definition}
\begin{definition}(Norm of metric derivative) \\
If $\rho_t \in \mathcal{P}^2(M)$, we call the metric derivative (or slope of metric derivative or speed) the function 
$$ \lvert \rho_t'\rvert  := \lim_{h\to 0} \frac{d_2(\rho_{t+h},\rho_t)}{h} $$
whenever it exists.
\end{definition}
\begin{lemma}(Metric derivative for $p$-absolutely continuous curves) \\
Let $\{\rho_t\}_{t\in [a,b]}$ be $p$-absolutely continuous. Then $\lvert \rho_s'\rvert$ exists Lebesgue a.e. and $ t \to \lvert \rho_t'\rvert $ is also $p$-integrable in $(a,b)$. 

\end{lemma}
\begin{proof}
As presented in \citep{Ambrosio}, letting $\{y_n\}$ be dense in $\{\rho_s\}_{s\in (a,b)}$ one can check that 
$$ \liminf_{t\to s} \frac{d(\rho_s,\rho_t)}{\lvert t-s \rvert} \geq \sup_n \liminf_{t \to s} \frac{\lvert d(y_n, \rho_s) - d(y_n,\rho_t)\rvert}{\lvert t-s\rvert} $$
from which the result follows. 
\end{proof}
\begin{lemma} \label{abscontleveltau}($\frac{1}{2}$-Hölder continuity of geodesic interpolation uniformly in $\tau$) \\
For $T>0$ if $\mathbf{Geo}_{\tau}(\{ \mu_{k}^{\tau} \}(t)$ denotes the geodesic interpolation (as in Remark \ref{interpolation}) on the interval $[0,T]$, then $ \{ \mathbf{Geo}_{\tau}(\{ \mu_{k}^{\tau} \}) (t)\}_{t \in [0,T]}$ is $\frac{1}{2}$-Hölder uniformly continuous, i.e. there exists $\Tilde{C} > 0$ independent of $\tau$ such that
\begin{equation}
    d_2(\mathbf{Geo}_{\tau}(\{ \mu_{k}^{\tau} \})(t),\mathbf{Geo}_{\tau}(\{ \mu_{k}^{\tau} \})(s)) \leq \Tilde{C}(t-s)^{1/2}.
\end{equation} 
\end{lemma} 
\begin{proof}
Let $\tau > 0$ be fixed, if $t_1,t_2 \in [k\tau, (k+1)\tau)$ and $ t_1 > t_2$ then the geodesic property of the exponential map yields: 
\begin{align}
  d_2(\mathbf{Geo}_{\tau}(\{ \mu_{k}^{\tau} \})(t_1),\mathbf{Geo}_{\tau}(\{ \mu_{k}^{\tau} \})(t_2)) &= \left(  \int \lvert \frac{t_1-t_2}{\tau} \nabla \phi_{k,k+1}^{c} \rvert^2 d \mu_{k}^{\tau} \right)^{1/2} \\
& = \left( \frac{t_1-t_2}{\tau} \right) d_2(\mu_{k}^{\tau}, \mu_{k+1}^{\tau}). 
\end{align}
Hence, by definition of the metric derivative we obtain
$$  \lvert \mathbf{Geo}_{\tau}(\{ \mu_{k}^{\tau} \}) '(t) \rvert = \lim_{h \to 0} \frac{(h/\tau)d_2(\mu_k^{\tau},\mu_{k+1}^{\tau})}{h} = \frac{d_2(\mu_k^{\tau},\mu_{k+1}^{\tau})}{\tau}. $$
With this calculation in mind, we compute using Hölder's inequality and Proposition \ref{squareest},
\begin{align*}
& d_2(\mathbf{Geo}_{\tau}(\{ \mu_{k}^{\tau} \})(t),\mathbf{Geo}_{\tau}(\{ \mu_{k}^{\tau} \})(s)) \\
 &= \int_s^t  \lvert \mathbf{Geo}_{\tau}(\{ \mu_{k}^{\tau} \}) '(r) \rvert dr \leq (t-s)^{1/2} \left( \sum_{k = 1}^{\infty} \frac{d_2(\mu_{k}^{\tau},\mu_{k+1}^{\tau})^2}{\tau} \right)^{1/2} \leq C (t-s)^{1/2} 
 \end{align*}
\end{proof}
Recall that by Corollary \ref{corollaryexists} we have ensured the existence of a limiting measure path (as $\tau \to 0$) of the geodesic interpolation of the minimizing movement scheme for $E_W$. We observe that uniform (on $\tau$) absolute continuity  (Lemma \ref{abscontleveltau}) implies absolute continuity of the limiting path.
\begin{corollary} (The limit shares the Hölder constant) \label{corolimit}\\
Let $T > 0$ and suppose that for every $t \in [0,T]$ we have that as $\tau \to 0$,  $ \mathbf{Geo}_{\tau}(\{ \mu_{k}^{\tau} \})(t) \xrightarrow{d_2} \mu(t)$, then $\mu (t)$ is $1/2$-Hölder continuous in $[0,T]$ with constant $C$. 
\end{corollary}
\begin{proof} Given $\epsilon > 0$, there exists $ \tau = \tau(\epsilon)$ such that $d_2(\mu(t),\mathbf{Geo}_{\tau}(\{ \mu_{k}^{\tau} \})(t)) < \frac{\epsilon}{2}$ and $d_2(\mu(s),\mathbf{Geo}_{\tau}(\{ \mu_{k}^{\tau} \})(s)) < \frac{\epsilon}{2}  $ from which applying the previous result (Lemma \ref{abscontleveltau}) and triangle inequality we obtain
$$ d_2(\mu(t), \mu(s)) \leq \epsilon + C(t-s)^{1/2}.  $$
Because $\epsilon$ is arbitrary and $C$ does not depend on $\tau$ we get the result. 
\end{proof}
Recall that if $(\mu_t,v_t)$ satisfies the continuity equation in the sense of distributions, then for every $ f \in C_c^{\infty}(M)$ 
$$ \frac{d}{dt}\int f(x) d\mu_t = - \int \langle \nabla f(x), v_t(x) \rangle_x d\mu_t(x). $$
See for example \cite{santambrogio} Proposition 4.2.
\begin{lemma}\label{Taylor}(Computation of the velocity field) \\
Letting $(M^n,g)$ be a smooth, connected, compact manifold, the velocity field of the geodesic interpolation of the minimizing movement scheme $\geo(\{\mu_{k}^{\tau})$ is given by parallel transporting the gradient of the $c$-transform of it's Kantorovich potential on each interval.
\end{lemma}
\begin{proof}
Suppose that $ \mu_t = \exp_x(tv(x))_{\#}\mu_0$ for some $\mu_0 \in \mathcal{P}_{ac}^2(M)$ and a differentiable map $ v: M \to TM$, then by compactness of $M$ and dominated convergence,
\begin{align*}
 \frac{d}{dt} \int f(x) d\mu_t(x) =  & \int \frac{d}{dt} f(\exp_x(tv(x)))d\mu_0(x)  \\
 &= \int \langle \nabla f(\exp_x(t v (x))), \Pi_{t,v(x)} ( v(x) )\rangle_{\exp_x(tv(x))} d\mu_0(x), 
\end{align*}
where $\Pi_{t,v} (\cdot) = (d\exp_x)_{tv}(\cdot) $ denotes parallel transport along the geodesic $ t \to \exp_x(tv)$. 
\\
Now denote $T_t^{\tau}(x) = \exp_x \left( (\frac{(k+1) \tau - t}{\tau}) \nabla \phi_{k,k+1}^c(x) \right)$ and let $v_t^{\tau}$ be such that 
$$ \frac{\partial T_{t}}{\partial t}(x) = v_t^{\tau}(T_{t}(x)) = (d\exp_x)_{-t \frac{\nabla \phi_{k,k+1}^c}{\tau}}\left(- \frac{\nabla\phi_{k,k+1}^c}{\tau}\right) =\Pi_{t,\frac{-\nabla\phi_{k,k+1}^c}{\tau}}\left(-\frac{\nabla \phi_{k,k+1}^c}{\tau}\right). $$ 
Because the differential of the exponential map at $0$ is the identity operator we get that by Taylor expansion for $t \in [k\tau,(k+1)\tau)$

$$ v_t^{\tau}(x) = \frac{\nabla \phi_{k,k+1}^c (x)}{\tau} + R_{t}^{\tau}(x), $$
where $R_{t}^{\tau} (x) \in T_xM$ and satisfies that as $ k \tau \to t$ (equivalently $\tau \to 0$) we have $ \lvert R_{t}^{\tau}\rvert_{x} \to 0$.
\end{proof}
\begin{definition}(First variation of a functional in $\mathcal{P}(M)$) \\ 
Let $F$ be a functional $F: \mathcal{P}_2(M) \to \mathbb{R} $, let $\rho \in \mathcal{P}_2(M)$ be fixed  and $\epsilon > 0$, for any $\Tilde{\rho} \in \mathcal{P}^2_{ac} \cap L^{\infty}(M)$, define $\nu = \Tilde{\rho} - \rho $, we say that $\frac{\delta F}{\delta \rho} (\rho)$ is the first variation of $F$ evaluated at $\rho$ if 

$$ \frac{d}{d \epsilon} \bigg \lvert_{\epsilon = 0} F(\rho + \epsilon \nu) = \int \frac{\delta F}{\delta \rho}(\rho) d\nu. $$

\end{definition}
\begin{theorem} (Optimality criteria)
\label{firstvar} \\
For a functional $F: \mathcal{P}_2(M) \to \mathbb{R} $ suppose that $\mu \in \argmin_{\nu \in \mathcal{P}_2(M)}  F(\nu)$. Assume that for every $\epsilon > 0$ and for every $\rho$ absolutely continuous with $L^{\infty}(M)$ density 
$$ F((1-\epsilon) \mu + \epsilon \rho) < \infty $$
Let $\Tilde{c}:= essinf\left\{\frac{\delta F}{\delta \rho} (\mu) \right\}$. If  $\frac{\delta F}{\delta \rho} (\mu)$ is continuous, then
\begin{equation} 
\frac{\delta F}{\delta \rho} (\mu)(x) \geq \Tilde{c} \: \: \forall x \in M,
\end{equation}
\begin{equation}
     \frac{\delta F}{\delta \rho} (\mu)(x) = \Tilde{c} \: \: \forall x \in \spt(\mu).
\end{equation}
\end{theorem}
The proof can be found as Theorem 7.20 in \cite{santambrogio}.
\begin{lemma} (Computation of first variations) \\
For each of the following cases let $\mu$ satisfy for each functional $F$ the hypothesis of the last theorem, then
\[ \begin{cases} 
      \frac{\delta F}{\delta \rho} (\mu) = \phi_{\mu,\nu}   \text{ if } F(\mu) = \frac{d_2^2(\mu,\nu)^2}{2}, \\ \\
     \frac{\delta F}{\delta \rho} (\mu) =  2 (W \ast \mu) \text{ if } F(\mu) = \int_M \int_M W(x,y)d\mu(x) d\mu(y),
   \end{cases}\]
where as before $\phi_{\mu,\nu}$ is the Kantorovich potential whose negative gradient pushes $\mu$ to $\nu$ optimally with respect to $d^2/2$.
\end{lemma}
\begin{proof}
The first computation can be found in \citep{santambrogio}, as a particular case of Proposition 7.16, while for the second one, note that
\begin{align*}
& F(\rho + \epsilon \nu) = \int_M \int_M W(x,y) d(\rho + \epsilon \nu)(x) d(\rho + \epsilon \nu)(y) \\
&= F(\rho) + \epsilon^2 F(\nu) + \epsilon \left( \int_M \int_M W(x,y) d\rho(x) d\nu(y) + \int_M \int_M W(x,y) d\rho(y) d\nu(x) \right). 
\end{align*}
where the result is obvious by dividing by $\epsilon$ and taking the limit. \\
Clearly if $W$ is symmetric, as in the case of assumptions \eqref{assumptions}, 
$$ \frac{\delta F}{\delta \rho }(\mu) = 2 \int_M W(x,y) d\mu (y) = 2 (W \ast \mu)  (x). $$
\end{proof}
The Kantorovich potentials are known to exist in general settings such as Polish spaces but the question of their regularity is usually more subtle (Chapter 10 \citep{VillaniOldAndNew}). We recall the concepts of semiconcavity/semiconvexity from non-smooth analysis for which we follow \citep{McCannCordero}.\\
In our context compactness of $M$ yields semiconcavity of both the Kantorovich potential and the convolution of the interaction which together yield differentiability of the infimal convolution as we will show in Lemma \ref{phidifferentiable}. 
\begin{definition} (Locally semi-concave) \\
Let $U \subseteq M$ be open, we say $f: U \to \mathbb{R}$ is semi-concave at $x_0$ if there exists a neighborhood of $x_0$ and a constant $ C \in \mathbb{R}$ such that for every $x \in U$ and $ v \in T_xM$
\begin{equation}
    \limsup_{r \to 0} \frac{f(\exp_x(rv)) + f(\exp_x(-rv)) - 2f(x) }{r^2} \leq C,
\end{equation}
where $\exp_x$ denotes the exponential at $x$.
\end{definition} 
\begin{remark} In \citep{McCannCordero} it is shown that semi-concave functions admit non-empty superdifferentials, which implies that semi-concavity together with semi-convexity yields differentiability. It is also shown there that $c$-concave functions are semi-concave (Proposition 3.14) and that $x \to d(x,y)^2$ is everywhere semi-concave but fails to be semi-convex at the cut locus (Proposition 2.5). We refer to \citep{McCannCordero} for details and proofs.
\end{remark}
\begin{lemma}\label{jointlysmooth} (Joint smoothness or Riemannian distance squared away from cut locus) \\
In the context of our smooth, connected compact manifold $(M,g)$, the square of Riemannian distance is smooth away from the cut locus, i.e. if $(x_0,y_0) \not \in Cut$ then $(x,y) \to d(x,y)^2$ is smooth in a neighborhood of $(x_0,y_0)$.
\end{lemma}
\begin{proof} This lemma can be understood as analogous to Theorem 3.6 in \citep{McCannLorenz} part c), so we follow the proof from there. Note that if $(x,y)$ is near $(x_0,y_0) $ because the cut locus is closed, there exists a ball around $(x_0,y_0)$ not intersecting the cut locus. Note that the function $(x,v) \to (x,\exp_x(v))$ acts as a diffeomorphism near $(x_0,v_0)$ where we define $ v_0 = \exp_{x}^{-1}(y_0)$. Note that $ t \to \exp_x(tv_0)$ is the geodesic joining $x$ and $y_0$ and not $x_0$ and $y_0$. By symmetry, $(y_0,w_0) \to \exp_y(w_0)$ acts as a diffemorphism where again $w_0 = \exp_{y}^{-1}(x_0)$. Given $ z \in T_xM$ denote by $z_* = \langle z, \cdot \rangle_x \in T_x^*M$ then note that $-D(d(x,y)) = (v_*/\lvert v_*\rvert_x, w_*/\lvert w_*\rvert_y)\lvert_{(v,w) = (\exp_x^{-1}(y),\exp_y^{-1}(x))} $ as one can see by noting that $v_0,w_0$ are the tangent vectors of the geodesic (that exists as $(x_0,y_0) \not \in Cut$) and the operator notation $v_*$ and $w_*$ turns $v,w$ to covectors so that they belong to $T^*xM$ and $T_y^*M$ respectively, (see again \citep{McCannLorenz}), from which we conclude all components depend smoothly on $(x,y)$ yielding the result.
\end{proof}
\begin{lemma} (Semiconcavity of convolution) \label{Wastmuissemiconcave} \\
Let $\mu \in \mathcal{P}_2(M)$ Borel, assume that $W(x,y)$ satisfies assumptions \ref{h0}-\ref{hisincreasing} with $r_h \leq t_{inf}$, i.e. 
$h$ is non-decreasing on $[t_{inf},diam(M)]$, then everywhere on $M$ the function $ x \to (W \ast \mu )(x) $ is semiconcave. 
\end{lemma}
\begin{proof}
Note that we can decompose the convolution: 
\begin{equation} \label{separateconvolution}
    W\ast \mu (x) = \int_{ \{y: d(x,y) \leq \sqrt{r_h} \}} W(x,y) d\mu(y) + \int_{\{y: d(x,y) > \sqrt{r_h}\}} W(x,y) d\mu(y).
\end{equation}
As in \citep{McCannCordero} if $ d(x_0,y) < t_{Cut(x_0)}$ then $ x \to d(x,y)^2$ is smooth at $x_0$ and so $ h(d(x,y)^2)$ is semiconcave as a $C^2$ function on a compact set has a bounded Hessian. Notice that the bound may depend on $y$ but $ (x,y) \to d(x,y)^2$ is jointly smooth away from the cut locus by Lemma \ref{jointlysmooth} so $ y \to \hess_xW(x,y)$ is continuous and because $M$ is complete, $ \{ y: d(x,y) \leq \sqrt{r_h} \}$ is compact and hence the $x$-Hessian of the first term in \eqref{separateconvolution} is uniformly upper bounded. \\
For the second term of \eqref{separateconvolution}, note that in the region $\{y: d(x,y) > \sqrt{r_h} \} $ the assumption \ref{hisincreasing} ensures the hypothesis of the Chain rule for supergradients (Lemma 5 in \citep{McCann}) is satisfied for $ x \to h(d(x,y)^2) $, as $ x \to d(x,y)^2$ is everywhere semiconcave, $h$ is twice differentiable there and $ r \to h(r)$ is non-decreasing on $ \{ r > r_h \}$, as in \citep{McCannCordero} Corollary 3.13 there exists $ C > 0$ such that for every $(x,y) \in M$ and $ u \in T_xM$ with $d(x,y) > c$,
\begin{equation} \label{finitedifr}
    \limsup_{r \to 0^+} \frac{W(\exp_x(ru),y) + W(\exp_x(-ru),y) - 2W(x,y)}{r^2} \leq C.
\end{equation}
Let us use the following notation:
\begin{equation} 
    f_r(x,y,u) : = \frac{W(\exp_x(ru),y) + W(\exp_x(-ru),y) - 2W(x,y)}{r^2}.
\end{equation}
To show semi-concavity of the convolution, we will use Fatou's Lemma, for which it is enough to show there exists a constant that bounds $\{ f_r(x,y,u)\}_{r > 0}$ in the region $\{y: d(x,y)\geq \sqrt{ r_h} \}$ so that we can take the limit superior (as $r \to 0$) inside of the integral.\\
As \eqref{finitedifr} holds for every $(x,y)$ in the given region, by Lemma 3.11 in \citep{McCannCordero} we obtain that there exists $r^*$ and a smooth function $V$ such that $ x \to W(x,y) + V(x)$ is actually geodesically $C$-concave for every $y \in \{y: d(x,y) \geq r_h \}$. When $ r < r^*$ by $C$-concavity (as $x$ is the midpoint between $\exp_x(ru)$ and $\exp_x(-ru)$, the definition yields $f_r(x,y,u) \leq C$ for all quadruples $(r,x,y,u) \in \{(r,x,y,u): 0 < r, d(x,y) > c, u \in T_xM \}$, so in this region, for fixed $x \in M$, the functions $ y \to f_r(x,y,u)$ satisfy $ f_r(x,y,u) \leq C$ and therefore using reverse Fatou's Lemma: 
\begin{eqnarray*}
&\displaystyle \limsup_{r \to 0}  \int_{ \{ y: d(x,y)^2 > r_h \}}  \frac{W(\exp_x(ru),y) + W(\exp_x(-ru),y) - 2W(x,y) }{r^2} d\mu(y) \\
& \hspace{1cm} \leq \displaystyle \int_{ \{ y: d(x,y)^2 > r_h \}} C d\mu(y) \leq C,
\end{eqnarray*}
which is semi-concavity of the convolution (with the same constant) as desired.
\end{proof}
\begin{lemma} (Differentiability of Kantorovich potentials in the whole support) \label{phidifferentiable} \\
Let $\phi_{k,k+1}$ be the Kantorovich potential from $\mu_{k}^{\tau}$ to $\mu_{k+1}^{\tau}$, then it's $c$-transform $\phi_{k,k+1}^{c}$ is differentiable at $x$ for every $ x \in \spt(\mu_{k+1}^{\tau})$. 
\begin{proof} By optimality (Lemma \ref{optimality}) we know that for every $ x \in \spt(\mu_{k+1}^{\tau})$
$$ \frac{\phi_{k,k+1}^{c}}{\tau}(x) + W \ast \mu_{k+1}^{\tau} (x)  \geq  \Tilde{c}  \text{ with equality on } \spt(\mu_{k+1}^{\tau}).  $$
By definition of the infimal convolution $\phi_{k,k+1}^c$ is $c-$concave and hence semiconcave as in Proposition 3.14 in \cite{McCannCordero}. \\
By assumptions \ref{h0}-\ref{hisincreasing} we apply Lemma \ref{Wastmuissemiconcave} to conclude that $ W \ast \mu_{k+1}^{\tau} (x)$ is semiconcave. \\
Hence, everywhere in $\spt(\mu_{k+1}^{\tau})$, $\phi_{k,k+1}^c = \tau( \Tilde{c} - W \ast \mu_{k+1}^{\tau}) $ is also semiconvex  meaning that  $\phi_{k,k+1}^{c}$ is both semiconcave and semiconvex and hence continuously differentiable at $x \in \spt(\mu^{\tau}_{k+1})$. 
\end{proof} 
\end{lemma}
\begin{lemma} \label{optimality}(First variations for the minimizing movement scheme) \\
Let $\{ \mu_k^{\tau} \}_{k \in \mathbb{N}}$ be the minimizing movement scheme with initial measure $\mu_0 \in \mathcal{P}_{ac}^2(M)$, let $\phi_{k,k+1}$ be the Kantorovich potential for which the exponential of it's negative gradient pushes $\mu_{k}^{\tau}$ onto $\mu_{k+1}^{\tau}$, then on the support of $\mu_{k+1}^{\tau}$
$$ -\frac{\nabla_M \phi_{k,k+1}^{c}}{\tau} =  \nabla_M (W \ast \mu_{k+1}^{\tau}).$$
\end{lemma}
\begin{proof}
To agree with notation, let $F(\nu) = \frac{d_2^2(\nu,\mu_{k}^{\tau})}{2\tau} + E(\nu) $. Theorem \eqref{firstvar} says

$$ \frac{\phi_{k,k+1}^{c}}{\tau} + W \ast \mu_{k+1}^{\tau} = \frac{\delta F}{\delta \rho}(\mu_{k+1}^{\tau}) = c \text{ on } \spt(\mu_{k+1}^{\tau}). $$
Using the previous lemma taking the gradient on both sides gives the result.
\end{proof}
\section{Finite speed of propagation and proof of the main theorem \ref{existenceanduniqueness}}
In this section we take a look at a consequence of the Euler-Lagrange condition that will ensure that the evolution of the measures is controlled. This proposition is key to ensure differentiability of the interaction potential needed to conclude convergence in the continuity equation. 
\begin{proposition}(Finite speed of propagation in the minimizing movement scheme) \label{finitespeed} \\
Given $ \tau > 0$ and $ \mu_k^{\tau} \in \mathcal{P}_2(M)$, let $ L > 0$ denote the Lipschitz constant of the potential from \ref{h0}-\ref{Wishd}, if 
$$ \mu_{k+1}^{\tau} \in \argmin_{\rho \in \mathcal{P}_2(M)} \left\{ \frac{1}{2} \int \int W(x,y) d\rho d\rho + \frac{1}{2\tau} d_2(\mu_k^{\tau},\rho)^2 \right\} $$
we have $$ \spt(\mu_{k+1}^{\tau}) \subseteq \{ x \in M: d(x,\spt(\mu_{k}^{\tau})) \leq L \tau \}. $$
\end{proposition}
\begin{proof}
By Lemma \ref{optimality} we know we can compute the gradient of $\phi_{k,k+1}^{c}$, which is defined $\mu_{k+1}^{\tau}$ everywhere on $\spt(\mu_{k+1}^{\tau})$ and hence the map $ x \to \exp_x(\nabla \phi_{k,k+1}^c (x))$ is well defined and supported in the subdifferential of a $c-$concave map $(\phi_{k,k+1}^c)$, by the converse in proposition 1.30 in \citep{Users} we get that this map is optimal and pushes $\mu_{k+1}^{\tau}$ to $\mu_k^{\tau}$ meaning that 
\begin{equation}\label{wassphi}
d_2(\mu_{k+1}^{\tau}, \mu_{k}^{\tau})^2 = \int_M \lvert \nabla \phi_{k,k+1}^{c} \rvert ^2 d\mu_{k+1}^{\tau}.
\end{equation} 
Now by assumptions \ref{h0}-\ref{Wishd} we know that (Remark \ref{WLip}) $ y \to W(x,y)$ is Lipschitz for every $ x \in \spt(\mu_{k+1}^{\tau})$, from which the convolution is Lipschitz with the same constant $L$, as $\mu_{k+1}^{\tau}$ is a probability measure, i.e.
$$ \lvert W \ast \mu_{k+1}^{\tau} (x_1) - W \ast \mu_{k+1}^{\tau} (x_2) \rvert  \leq L d(x_1,x_2).$$
Consequently, because the norm of the gradient is bounded by the metric derivative (see Chapter 1 in \citep{Ambrosio}), we obtain that  for every $x \in \spt(\mu_{k+1}^{\tau})$
\begin{equation} \label{10bis}
\lvert \nabla_M (W \ast \mu_{k+1}^{\tau}) (x) \rvert_x \leq L .
\end{equation}
But by Lemma \ref{optimality} this means that 
$$ d(x,\exp_x(\nabla \phi_{k,k+1}^{c}(x))) = \lvert \nabla \phi_{k,k+1}^{c}(x) \rvert \leq L \tau .$$
Note that $L$ is the global Lipschitz constant of the interaction potential and hence independent of $\tau$. Consequently every point $x \in \spt(\mu_{k+1}^{\tau})$ is transported a distance of at most $L\tau$ from which triangle inequality yields the result. 
\end{proof}
\begin{corollary}
(Small time differentiability of the potential) \label{SmallTime} \\
Fix $\tau > 0$, let $\mu_0 \in \mathcal{P}_{ac}^2(M)$ with $\delta_{\mu_0} > 0$, where $\delta_{\mu_0}$ is the distance to cut from Definition \ref{distancecut}, then for all $ t \in [0, \lfloor\frac{\delta_{\mu}}{2L \tau} \rfloor]$ the function $ x \to W(x,y)$ is differentiable in the support of $\geo(\{\mu_k^{\tau}\})(t)$ where $\mu_{k}^{\tau}$ is the minimizing movement scheme at level $\tau$ defined in \eqref{minimizingmovements}.
\end{corollary}
\begin{proof}
Observe that the finite speed of propagation (Proposition \ref{finitespeed})  immediately ensures that if $\mu_0 \in \mathcal{P}_{ac}(M)$ with $ \delta_{\mu_0} > 0$ then
$$ \spt(\mu_{1}^{\tau}) \subseteq \{ x \in M: d(x,\spt(\mu_0)) \leq L\tau \}. $$
Consequently, as we have seen in Lemma \ref{phidifferentiable} the map $ x \to \exp_x( \nabla \phi_{k,k+1}^c(x))$ is well defined and is an optimal transport map (with cost $d^2/2$) which means that we can apply finite speed of propagation (Lemma \ref{finitespeed}) at every $k$ and hence, 
$$ \spt(\mu_{k}^{\tau}) \subseteq \{x \in M: d(x,\spt(\mu_0)) \leq  Lk \tau \}. $$
For $k \in \mathbb{N}$, consider $(x,y) \in \spt(\mu_0)^2, (x_k,y_k) \in \spt(\mu_{k}^{\tau})^2$ and $(x',y') \in Cut$.
By definition of $\delta_{\mu_0}$, 
\begin{equation*}
    d(x,x') + d(y,y') \geq \delta_{\mu_0}
\end{equation*}
Using the triangle inequality twice, 
\begin{equation}
    d(x_k,x') + d(x_k,x) + d(y_k,y') + d(y_k,y) \geq \delta_{\mu_0}.
\end{equation}
By finite speed of propagation $2k$-times, 
\begin{equation*}
    2kL\tau + d(x_k,x') + d(y_k,y') \geq \delta_{\mu_0}.
\end{equation*}
Now using definition \ref{distancecut}, because $\delta_{\mu^{\tau}_k}$ is an infimum, $(x_k,y_k)$ and $(x',y')$ are arbitrary, 
\begin{equation}
    \delta_{\mu_{k}^{\tau}} \geq \delta_{\mu_0} - 2k \tau L. 
\end{equation}
Hence, as long as $\delta_{\mu_0} - 2k\tau L > 0$, the geodesic interpolation \ref{interpolation} guarantees that  $x \to W(x,y)$ is differentiable in the support of the measures up to $\mu_{k}^{\tau}$. \\
Notice that $\delta_{\mu_0} - 2k\tau L > 0 $ occurs exactly when $ k < \delta_{\mu_0}/(2L\tau)$ as desired.
\end{proof}
\begin{lemma} (Contraction of Wasserstein distances for product measures) \label{contractd1} \\
Let $\mu,\nu \in \mathcal{P}_1(M)$, denote by $\mu \otimes \mu$ and $\nu \otimes \nu$ the product measures on $M \times M$, then
$$ d_1(\mu \otimes \mu ,\nu \otimes \nu) \leq 2 d_1(\mu , \nu) $$
\end{lemma}
\begin{proof}
Note that if $\pi \in \Pi(\mu,\nu)$ then $\pi \otimes \pi \in \Pi( \mu \otimes \mu, \nu \otimes \nu)$ and note that 
$$ \int_M d(x,y) d\pi(x,y) + \int_M d(\tilde{x},\tilde{y}) d\pi(\tilde{x},\tilde{y}) = \int_{M \times M} d_{M \times M} ((x,\tilde{x}),(y,\tilde{y})) d \pi \otimes \pi (x,\tilde{x},y,\tilde{y}) $$
from which taking infimum on both sides yields the result. 
\end{proof}
\subsection{Evaluation of the limt}
In this section we prove Theorem \ref{existenceanduniqueness}, the goal is to show that the limiting measure path from Corollary \ref{corollaryexists} satisfies the continuity equation. The idea is to use the $d_2$ convergence together with the finite speed of propagation to ensure $ x \to W(x,y)$ is differentiable in the whole support of the measures at level $\tau$ so that we can differentiate inside of the convolution term in \eqref{aggregation}. 
\subsection{Proof of Theorem \ref{existenceanduniqueness}}
\begin{proof}
Assumptions \ref{assumptions} on $E_W$  ensure the hypothesis of Corollary \ref{corollaryexists} are satisfied in $(\mathcal{P}_2(M), d_2)$ which ensures the existence of a limiting path $ \{\mu (t) \}_{t \in [0,\infty)}$ for the family of geodesic interpolations  $\mathbf{Geo}_{\tau}(\{\mu_{k}^{\tau} \}) (t)$.\\
This interpolation satisfies the continuity equation with $v_t^{\tau}$ given as in Remark \eqref{Taylor}, by Lemma \ref{firstvar} we replace the vector field with $\nabla (W \ast \mu_{k+1}^{\tau})$ and finally  for $ f \in C^{\infty}_c([0,T]\times M)$ we aim to compute
$$ \lim_{\tau \to 0}\left( \underbrace{\int_0^{T} \int_M \partial_t f(x,t)d\mathbf{Geo}_{\tau}(\{\mu_{k}^{\tau} \}) (t)dt}_{(\romannumeral 1)} +\underbrace{\int_0^T \int_M \langle \nabla f(x,t), v_t^{\tau} \rangle_x d\mathbf{Geo}_{\tau}(\{\mu_{k}^{\tau} \}) (t)dt}_{(\romannumeral 2)} \right) $$
The limit of $(\romannumeral 1)$: notice that $ \partial_tf(x,t)$ is continuous on $x$ so by $\mathbf{Geo}_{\tau}(\{\mu_{k}^{\tau} \}) (t) \xrightarrow{d_2} \mu(t) $ as $ \tau \to 0$: 
 $$ \int_M \partial_t f(x,t) d\mathbf{Geo}_{\tau}(\{\mu_{k}^{\tau} \}) (t) \xrightarrow{\tau \to 0} \int_M \partial_t f(x,t) d\mu(t) \hspace{0.5cm } \forall t\in [0,T].$$
 Because $[0,T]$ is compact, the Dominated Convergence Theorem yields
 $$ \int_0^{T} \int_M \partial_t f(x,t) d\mathbf{Geo}_{\tau}(\{\mu_{k}^{\tau} \}) (t) dt \xrightarrow{\tau \to 0} \int_0^T \int_M \partial_t f(x,t) d\mu(t) dt. $$
To analyze $(\romannumeral 2)$, denote $T^{\tau,k}_t(x) = \exp_x \left(  \frac{((k+1)\tau - t)}{\tau} \nabla\phi_{k,k+1}^c(x)  \right)  $, with this notation $(T_t^{\tau,k})_{\#}\mu_{k+1}^{\tau} = \geo(\{\mu_{k}^{\tau}\})(t) $ and so by definition and using the observation of Lemma \ref{Taylor}, if $N_{\tau} = \lfloor T/\tau \rfloor $, 
\begin{align}\label{romannumeral2}
\begin{split}
& (\romannumeral 2) = \int_0^T \int_{M} \langle \nabla_M f(x,t), v_t^{\tau}(x) \rangle_x d\geo(t) dt \\ &= \sum_{k=0}^{N_{\tau}} \int_{k\tau}^{(k+1)\tau} \int_{M} \langle \nabla_M f(T_t^{\tau,k}(x),t), \Pi_{t,\gamma_{k,\tau}} \left( -\frac{\nabla \phi_{k,k+1}^c (x)}{\tau} \right) \rangle_{T_t^{\tau,k}(x)} d\mu_{k+1}^{\tau}(x) dt  \\
& \hspace{2 cm} - \int_{T}^{N_{\tau}+1} \int_{M} \langle \nabla_M f(T_t^{\tau,k}(x),t), \Pi_{t,\gamma_{k,\tau}} \left( -\frac{\nabla \phi_{k,k+1}^c (x)}{\tau} \right) \rangle_{T_t^{\tau,k}(x)} d\mu_{k+1}^{\tau}(x) dt \\
& = -  \sum_{k=0}^{N_{\tau}} \int_{k\tau}^{(k+1)\tau}\int_M \langle \Pi_{t,\gamma_{k,\tau}}^{-1}\left( \nabla_M f(T_t^{\tau,k}(x),t) \right), \frac{\nabla \phi_{k,k+1}^c(x)}{\tau} \rangle_x d\mu_{k+1}^{\tau}(x) dt \\
&\hspace{2 cm} + \int_{T}^{N_{\tau}+1}\int_M \langle \Pi_{t,\gamma_{k,\tau}}^{-1}\left( \nabla_M f(T_t^{\tau,k}(x),t) \right), \frac{\nabla \phi_{k,k+1}^c(x)}{\tau} \rangle_x d\mu_{k+1}^{\tau}(x) dt,
\end{split}
\end{align}
where as in Lemma \ref{Taylor} $\Pi_{t,\gamma_{k,\tau}}$ denotes parallel transport along the curve $\gamma_{k,\tau}: [\tau k, \tau(k+1)) \to M $ given by $  \gamma_{k,\tau}(t) = T_{t}^{\tau,k}$ which satisfies $\gamma_{k,\tau} ((k+1)\tau)  = x  $.  The second equality follows from the fact that parallel transport is an isometry.  \\
Focusing first in the inner most integral of the first term observe that by adding and subtracting the same term we can write
\begin{align}\label{sumzero}
\begin{split}
 &  \int_M \langle \Pi_{t,\gamma_{k,\tau}}^{-1}\left( \nabla_M f(T_t^{\tau,k}(x),t) \right), \frac{\nabla \phi_{k,k+1}^c(x)}{\tau} \rangle_x d\mu_{k+1}^{\tau}(x) \\
   & \hspace{1 cm} =  \underbrace{\int_M \langle \Pi_{t,\gamma_{k,\tau}}^{-1}\left( \nabla_M f(T_t^{\tau,k}(x),t) \right) - \nabla_Mf(x,t), \frac{\nabla \phi_{k,k+1}^c(x)}{\tau} \rangle_x d\mu_{k+1}^{\tau}(x)}_{A}  \\ & \hspace{1 cm} +   \underbrace{\int_M \langle  \nabla_Mf(x,t), \frac{\nabla \phi_{k,k+1}^c(x)}{\tau} \rangle_x d\mu_{k+1}^{\tau}(x)}_{B}.
   \end{split}
\end{align}
We start by obtaining uniform (on $\tau$) bounds for $A$, note that by Cauchy-Schwarz, for every $x \in \spt(\mu_{k+1}^{\tau})$
\begin{align*}
&\langle \Pi_{t,\gamma_{k,\tau}}^{-1}\left( \nabla_M f(T_t^{\tau,k}(x),t) \right) - \nabla_Mf(x,t), \frac{\nabla \phi_{k,k+1}^c(x)}{\tau} \rangle_x \\
\hspace{2cm } & \leq \lvert \Pi_{t,\gamma_{k,\tau}}^{-1}\left( \nabla_M f(T_t^{\tau,k}(x),t) \right) - \nabla_Mf(x,t) \rvert_x \bigg \lvert  \frac{\nabla \phi_{k,k+1}^c(x)}{\tau} \bigg \rvert_x . 
\end{align*}
But  the time derivative of $\Pi_{t,\gamma_{k,\tau}}^{-1}$ along integral curves gives the covariant derivative along $\dot{\gamma_{k,\tau}}$, so by Mean Value Theorem and Proposition \ref{finitespeed}, 
\begin{align*}
 \langle \Pi_{t,\gamma_{k,\tau}}^{-1} & \left( \nabla_M  f(T_t^{\tau,k}(x),t) \right)   - \nabla_M f(x,t), \frac{\nabla \phi_{k,k+1}^c(x)}{\tau} \rangle_x \\
\hspace{2cm } & \leq \  \sup_{\xi \in M} \lvert \nabla_{\dot{\gamma}_{k,\tau}}(\nabla_M f(\xi,t))\rvert_{\xi} ((k+1)\tau - t) \bigg \lvert  \frac{\nabla \phi_{k,k+1}^c(x)}{\tau} \bigg \rvert_x  \\
\hspace{2cm} & \leq \sup_{t\in [0,T]} \sup_{\xi \in M} \lvert \hess f_t \rvert_{\xi} \lvert \dot{\gamma}_{\tau,k}(\xi)\rvert_{\xi} ((k+1)\tau - t) \bigg \lvert  \frac{\nabla \phi_{k,k+1}^c(x)}{\tau} \bigg \rvert_x \\
& \leq ((k+1)\tau - t) L^2  \sup_{t \in [0,T]} \sup_{\xi \in M} \lvert \hess f_t \rvert_{\xi},
\end{align*}
where $f_t(\cdot) = f(\cdot,t)$ and in the last bound we used the fact that $\gamma_{k\tau}$ is a geodesic so the norm of it's tangent vector is constant, therefore bounded by Proposition \ref{finitespeed}. \\
With this bound in hand, going back to \eqref{sumzero} we find that 
\begin{align*}
 & \sum_{k=0}^{N_{\tau}} \int_{k\tau}^{(k+1)\tau}\int_M \langle  \Pi_{t,\gamma_{k,\tau}}^{-1}\left( \nabla_M f(T_t^{\tau,k}(x),t) \right) - \nabla_Mf(x,t), \frac{\nabla \phi_{k,k+1}^c(x)}{\tau} \rangle_x d\mu_{k+1}^{\tau}(x) dt \\
 & \leq  \sup_{t \in [0,T]} \sup_{\xi \in M} \lvert \hess f_t \rvert_{\xi} \:  L^2 \: \sum_{k=0}^{N_{\tau}} \int_{k\tau}^{(k+1)\tau}  ((k+1)\tau - t) dt \\
 & \leq \sup_{t \in [0,T]} \sup_{\xi \in M} \lvert \hess f_t \rvert_{\xi} \:  L^2 \: \left(\sum_{k=0}^{N_{\tau}} (k+1)\tau^2 - k\tau^2 - \frac{\tau^2}{2} \right) \\
 & = \frac{1}{2} \sup_{t \in [0,T]} \sup_{\xi \in M} \lvert \hess f_t\rvert_{\xi} \:  L^2 \: (N_{\tau}+1)\tau^2  = C  \:  (N_{\tau}+1)\tau^2.
\end{align*}
Because $ N_{\tau} \tau \to T$ as $ \tau \to 0$, this upper bound goes to $0$ as $\tau \to 0$, meaning that $A$ vanishes in the $\tau \to 0$ limit. \\
To study $B$ from \eqref{sumzero} note that using Lemma \ref{optimality}
\begin{align*}
   &  \int_M \langle  \nabla_Mf(x,t), \frac{\nabla \phi_{k,k+1}^c(x)}{\tau} \rangle_x d\mu_{k+1}^{\tau}(x) =  \int_{M} \langle \nabla_Mf(x,t), \nabla (W \ast \mu_{k+1}^{\tau})  \rangle_x d\mu_{k+1}^{\tau}(x).
\end{align*}
Hence our goal to finish the proof is to show that 
\begin{align} \label{divergenceterm}
 &\sum_{k=0}^{N_{\tau}} \int_{k \tau}^{(k+1)\tau} \int_M \langle \nabla_M f(x,s), \nabla_M (W \ast \mu_{k+1}^{\tau} \rangle_x d\mu_{k+1}^{\tau} ds \nonumber \\
 & \hspace{2 cm} \xrightarrow{\tau \to 0} \int_0^{T} \int_M \langle \nabla_M f(x,t), \nabla_M (W \ast \mu(s) ) \rangle_x d\mu(s) ds.
 \end{align}
 By density in the space of smooth functions we may assume without loss of generality that $ f(x,t) = \phi(x) a(t)$ from which we can change the order of integration in the left handside of \eqref{divergenceterm} and by mean value theorem for integrals to rewrite 
 \begin{align} \label{mvtsimplif}
 &\sum_{k=0}^{N_{\tau}} \int_{k \tau}^{(k+1)\tau} \int_M \langle \nabla_M f(x,s), \nabla_M (W \ast \mu_{k+1}^{\tau}) \rangle_x d\mu_{k+1}^{\tau} ds \nonumber \\
 & \hspace{2 cm} = \sum_{k=0}^{N_{\tau}} a(t^*_k) \int_M \langle \nabla_M \phi, \nabla_M (W \ast \mu_{k+1}^{\tau}) \rangle_x d\mu_{k+1}^{\tau},
 \end{align}
 where $t_k^{*} \in [k\tau),(k+1)\tau]$. 
 Observe that this limit concludes the proof as the extra term in \eqref{romannumeral2} vanishes in the $ \tau \to 0$ limit.\\ By Riemann integrability together with \eqref{mvtsimplif}, to establish the limit \eqref{divergenceterm} it is enough to show that 
 \begin{align}
  &   \sum_{k=0}^{N_{\tau}} a(t_k^*) \left( \int_M \langle \nabla_M \phi, \nabla_M (W \ast \mu_{k+1}^{\tau})\rangle_x d\mu_{k+1}^{\tau} \right. - \nonumber \\ & \hspace{3 cm}  \left. \int_M \langle \nabla_M \phi, \nabla_M (W \ast \mu (t_k^*))\rangle_xd\mu(t_k^*) \right) \to 0 .
 \end{align}
 Observe that by Fubini's theorem,  the term in parenthesis can be recast as 
 \begin{align*}
  &\int_M \langle \nabla_M \phi, \nabla_M (W \ast \mu_{k+1}^{\tau})\rangle_x d\mu_{k+1}^{\tau} - \int_M \langle \nabla_M \phi, \nabla_M (W \ast \mu (t_k^*))\rangle_xd\mu(t_k^*) \nonumber \\
  & =  \int_M \langle \nabla_M \phi, \nabla_M W (x,y) \rangle_x d (\mu_{k+1}^{\tau} \otimes \mu_{k+1}^{\tau}) - (\mu_{t_k^{*}} \otimes \mu_{t_k^{*}}) (x,y) \nonumber \\
 & \leq d_1(\mu_{k+1}^{\tau} \otimes \mu_{k+1}^{\tau}, \mu_{t_k^{*}} \otimes \mu_{t_k^{*}}) \leq d_1(\mu_{k+1}^{\tau}, \mu_{t_k^{*}})) \leq d_2(\mu_{k+1}^{\tau}, \mu_{t_k^{*}}),
 \end{align*}
 where the first inequality follows continuity of the integrand and the fact that $M \times M$ is compact together with the Kantorovich-Rubenstein Theorem (see \citep{McCannCordero}), the second inequality by Lemma \ref{contractd1}. Coming back to showing \eqref{divergenceterm}, we get 
 \begin{align*}
       &  \sum_{k=0}^{N_{\tau}} a(t_k^*) \left( \int_M \langle \nabla_M \phi, \nabla_M (W \ast \mu_{k+1}^{\tau})\rangle_x d\mu_{k+1}^{\tau} \right. - \nonumber \\ & \hspace{3 cm}  \left. \int_M \langle \nabla_M \phi, \nabla_M (W \ast \mu (t_k^*))\rangle_xd\mu(t_k^*) \right) \\ & \lesssim  \lvert \lvert a\rvert \rvert_{\infty}  \tau ( N_{\tau} + 1) \sup_{t \in [0,T]} d_2( \mu_{k+1}^{\tau},\mu(t)) \\ &\lesssim \tau ( N_{\tau} + 1) \sup_{t \in [0,T]} \{ d_2(\geo(\{\mu_{k}^{\tau})\}(t),\mu(t)) ) + d_2\geo(\{\mu_{k}^{\tau}\})(t),\mu_{k+1}^{\tau}) \},
 \end{align*}
 which shows that the sum in \eqref{divergenceterm} vanishes in the limit because $ \tau N_{\tau} \to 1$ together with the uniform limit \eqref{uniformont} and \eqref{abscontleveltau}.
\end{proof}
\section{Conclusions and extensions}
This work proved the existence of small time solutions for \eqref{aggregation} via the minimizing movement scheme under suitable conditions on the interaction potential. The assumption of dependence on Riemannian distance make it completely intrinsic and suitably general. A first line of investigation could be to derive long-time existence and geometry of solutions of the model for specific interaction potentials like power laws, for example. In this context, is it possible to reproduce the aggregation results from \citep{McCannCameron} in curved geometries? \\
Another interesting extension is the performance of numeric algorithms based on entropic optimal transportation compared to the usual PDE approximation methods. \\
The idea of using the minimizing movement scheme motivated from the seminal work \citep{Ambrosio} required an analysis of the optimality condition as there was no global $\lambda$-convexity of the functional. One way to avoid this problem is to restrict the geometry of the manifold to satisfy a non-negative cross-curvature condition which allows the set of optimal transport maps to be convex yielding $\lambda$-convexity. This approach enables the machinery of \citep{Ambrosio} which not only ensures existence and uniqueness but provides error bounds on the discrete approximation. \\
This work has shown that techniques from non-smooth analysis allow the Euler-Lagrange equation to imply enough regularity to solve the aggregation equation in small times. Further work should concentrate on the several possibilities: non-existence, non-uniqueness and concentration. The minimizing movement scheme is independent of the differentiability of the potential, giving us a candidate for a solution of the aggregation equation. The characterization of these possibilities is an open question in Riemannian manifolds, resolved for power laws in Euclidean setting in \citep{McCannCameron}.

\end{document}